\documentclass[11pt]{article}
\usepackage{amsmath,amsthm,amssymb}
\usepackage[margin=1in]{geometry}
\usepackage{enumitem}
\setlist[enumerate]{label=(\arabic*),itemsep=2pt,topsep=4pt,leftmargin=2.5em}
\numberwithin{equation}{section}

\theoremstyle{plain}
\newtheorem{theorem}{Theorem}[section]
\newtheorem{lemma}[theorem]{Lemma}

\theoremstyle{definition}
\newtheorem{definition}[theorem]{Definition}
\newtheorem{example}[theorem]{Example}
\theoremstyle{remark}

\title{\textbf{A Complete Characterization of Realizable (Embedding Dimension, Multiplicity) Pairs for Complete Intersection Numerical Semigroups}}
\author{Minglang Li and Yizhi Zhang}
\date{\today}

\begin{document}

\maketitle
\begin{abstract}
We give a complete characterization of the positive integer pairs $(e,m)$ that occur as the (embedding dimension, multiplicity) of a complete intersection numerical semigroup, thereby settling an open problem in the theory of relations of numerical semigroups. We prove that there exists a complete intersection numerical semigroup $\Gamma$ with $e(\Gamma)=e$ and $m(\Gamma)=m$ if and only if $(e,m)=(1,1)$ or $e\ge 2$ and $m\ge 2^{e-1}$.
\end{abstract}

\noindent\textbf{Keywords:} complete intersection; numerical semigroup; embedding dimension; multiplicity; gluing; semigroup ring

\noindent\textbf{MSC(2020):} 20M14 (primary); 13A02, 13D02 (secondary)

\section{Introduction}\label{sec:intro}

Let $\mathbb{N}$ denote the set of nonnegative integers. A numerical semigroup $\Gamma$ is a cofinite additive submonoid of $(\mathbb{N},+)$ containing $0$; equivalently, $\Gamma=\langle A\rangle$, where $A\subset\mathbb{N}$ is a finite set with $\gcd(A)=1$. Every numerical semigroup $\Gamma$ has a unique minimal generating set, denoted by $\mathrm{Msg}(\Gamma)$. The embedding dimension $e(\Gamma)=|\mathrm{Msg}(\Gamma)|$ and the multiplicity $m(\Gamma)=\min(\Gamma\setminus\{0\})$ are two fundamental invariants of $\Gamma$.

Let $A=\{a_1,\dots,a_n\}$ be the minimal generating set of $\Gamma$. The map $\varphi\colon\mathbb{N}^n\to\Gamma$, $(c_1,\dots,c_n)\mapsto c_1a_1+\cdots+c_na_n$, is surjective; a relation among the generators of $\Gamma$ (with respect to $A$) is a pair $(u,v)\in\mathbb{N}^n\times\mathbb{N}^n$ satisfying $u_1a_1+\cdots+u_na_n=v_1a_1+\cdots+v_na_n$; it records two different representations of the same element. All such relations form a congruence on $\mathbb{N}^n$; if a set of relations generates this congruence, it is called a presentation of $\Gamma$, and a presentation with the fewest relations is called minimal. Let $\rho(\Gamma)$ denote the number of relations in a minimal presentation of $\Gamma$, i.e., the minimum of the number of relations over all presentations. If $\rho(\Gamma)=e(\Gamma)-1$, then $\Gamma$ is called a complete intersection; equivalently, the local semigroup ring $R_\Gamma=k[\![\Gamma]\!]$ over a field $k$ is a complete intersection ring, i.e., its minimal free resolution is the Koszul complex.

Moscariello and Sammartano posed the following problem (Problem 29) in their survey \cite{MS24}, attributing it to Question 7.5 of Elmacioglu--Hilmer--O'Neill--Okandan--Park-Kaufmann \cite{EHOOPK24}: characterize all pairs $(e,m)$ of positive integers for which there exists a complete intersection numerical semigroup $\Gamma$ with $e(\Gamma)=e$ and $m(\Gamma)=m$.

In this paper we study the following problem: which pairs $(e,m)$ of positive integers are realized by a complete intersection numerical semigroup, i.e., for which pairs does there exist a complete intersection numerical semigroup $\Gamma$ with $e(\Gamma)=e$ and $m(\Gamma)=m$?

The problem rests on the following classical results.
\begin{theorem}\label{thm:delorme}(\cite[Proposition 9]{Del76})
A numerical semigroup is a complete intersection if and only if it is $\mathbb{N}$, or it is obtained by gluing two complete intersection numerical semigroups.
\end{theorem}
The standard facts on gluings can be found in \cite[\S 8.3]{RGS09} and \cite[\S 1]{AGS13}.

\begin{lemma}\label{lem:lower-bound}(\cite[Proposition 5]{AGS13})
Every complete intersection numerical semigroup satisfies $m(\Gamma)\ge 2^{e(\Gamma)-1}$, and this bound is attained.
\end{lemma}

\begin{theorem}\label{thm:rosales}(\cite[Theorem 5]{Ros01})
There exists a symmetric numerical semigroup with $e(\Gamma)=e$ and $m(\Gamma)=m$ if and only if $2\le e\le m-1$ or $(e,m)\in\{(1,1),(2,2)\}$.
\end{theorem}
The answer for complete intersections is much simpler than for symmetric semigroups: as soon as $m\ge 2^{e-1}$, every multiplicity is realizable.

Our complete characterization is the following.
\begin{theorem}\label{thm:main}
There exists a complete intersection numerical semigroup $\Gamma$ with $e(\Gamma)=e$ and $m(\Gamma)=m$ if and only if $(e,m)=(1,1)$ or $e\ge 2$ and $m\ge 2^{e-1}$. Equivalently, the set of realized pairs is

\[
\{(1,1)\}\cup\{(e,m)\in\mathbb{Z}^2 : e\ge 2,\ m\ge 2^{e-1}\}.
\]
\end{theorem}

This is the main result of the paper; its proof is given in Section~\ref{sec:main}. The proof is as follows. The necessity follows immediately from the lower bound $m(\Gamma)\ge 2^{e(\Gamma)-1}$. The sufficiency is constructive: first we construct a family of covering semigroups $C_e(x)$; for every $e\ge 2$ and every $x\ge 2^e$, the semigroup $C_e(x)$ is a complete intersection numerical semigroup of embedding dimension $e$ and multiplicity $2^{e-1}$ such that $x$ is not a minimal generator of $C_e(x)$; then we take $U=C_{e-1}(m)$ and use gluing to build $\Gamma=m\mathbb{N}+nU$, where $n$ is a suitable integer coprime to $m$. Both steps use only the basic facts about gluings.

The paper is organized as follows. Section~\ref{sec:prelim} fixes the notation and collects the preliminary facts, including the definitions of complete intersection and of gluing, and Delorme's gluing characterization. Section~\ref{sec:main} proves the main results: we first construct the covering family (Lemma~\ref{lem:cover}), then state the lower bound (Lemma~\ref{lem:lower-bound}), and finally give the proof of the main theorem; we also list a number of explicit examples. Section~\ref{sec:discussion} discusses the sharpness of the characterization, compares it with the symmetric case, and lists some open problems.

\section{Notation and preliminaries}\label{sec:prelim}

We use the standard notation of the theory of numerical semigroups. The minimal generating set $\mathrm{Msg}(\Gamma)$, the embedding dimension $e(\Gamma)$, and the multiplicity $m(\Gamma)$ of a numerical semigroup $\Gamma$ are defined in Section~\ref{sec:intro}. If $A\subseteq\mathbb{N}$ and $\gcd(A)=1$, then $\langle A\rangle$ is a numerical semigroup; conversely, every numerical semigroup is of the form $\langle A\rangle$.
\begin{definition}\label{def:ci}
Let $\Gamma$ be a numerical semigroup, and let $\rho(\Gamma)$ be the number of relations in a minimal presentation of $\Gamma$. If $\rho(\Gamma)=e(\Gamma)-1$, then $\Gamma$ is called a complete intersection. Equivalently, the local semigroup ring $R_\Gamma=k[\![\Gamma]\!]$ over a field $k$ is a complete intersection ring, i.e., its minimal free resolution is the Koszul complex.
\end{definition}

We now state the definition of gluing and its basic properties; they are the basis of all our constructions.
\begin{definition}\label{def:gluing}
Let $S_1,S_2$ be numerical semigroups with minimal generating sets $\mathrm{Msg}(S_1)$ and $\mathrm{Msg}(S_2)$. Let $a\in S_2\setminus\mathrm{Msg}(S_2)$ and $b\in S_1\setminus\mathrm{Msg}(S_1)$ with $\gcd(a,b)=1$. The numerical semigroup $S=aS_1+bS_2=\{ax+by : x\in S_1,\ y\in S_2\}$ is called the gluing of $S_1$ and $S_2$ at $(a,b)$.
\end{definition}
\begin{lemma}\label{lem:gluing}
Let $S=aS_1+bS_2$ be a gluing. Then:

\begin{enumerate}
\item $\mathrm{Msg}(S)=a\,\mathrm{Msg}(S_1)\cup b\,\mathrm{Msg}(S_2)$;
\item $e(S)=e(S_1)+e(S_2)$, and $m(S)=\min(a\,m(S_1),b\,m(S_2))$;
\item $S$ is a complete intersection if and only if both $S_1$ and $S_2$ are complete intersections;
\item conversely, every complete intersection numerical semigroup other than $\mathbb{N}$ is a gluing of two complete intersection numerical semigroups.
\end{enumerate}

\begin{proof}
Write $A_1=\mathrm{Msg}(S_1)$ and $A_2=\mathrm{Msg}(S_2)$.
\begin{enumerate}
\item Since $S=aS_1+bS_2=\langle aA_1\cup bA_2\rangle$, it suffices to show that no element of $aA_1\cup bA_2$ can be written as a sum of two positive elements of $S$. To this end we first establish the following divisibility fact. Fix $a_i\in A_1$; if $aa_i=ax+by$ ($x\in S_1$, $y\in S_2$), then $y=0$ and $x=a_i$. Indeed, from $by=a(a_i-x)$ and $\gcd(a,b)=1$ we get $a\mid y$; write $y=at$ ($t\ge 0$). If $t=0$, then $y=0$, hence $x=a_i$. If $t\ge 1$, then $a_i=x+bt$. If $x\ge 1$, then $a_i$ is a sum of two positive elements of $S_1$, contradicting $a_i\in A_1$; if $x=0$, then $a_i=bt$, and since $b\notin\mathrm{Msg}(S_1)$, the element $b$ is a sum of two positive elements of $S_1$, hence so is $bt$, again contradicting $a_i\in A_1$. Thus $y=0$ and $x=a_i$. Symmetrically, for every $b_j\in A_2$, if $bb_j=ax+by$, then $x=0$ and $y=b_j$. Now suppose that $aa_i=u+v$ with $u,v\in S\setminus\{0\}$. Write $u=ax_1+by_1$ and $v=ax_2+by_2$. By the fact above, $y_1=y_2=0$ and $x_1+x_2=a_i$, so $u=ax_1$ and $v=ax_2$ ($x_1,x_2\ge 1$), contradicting $a_i\in A_1$; hence $aa_i\in\mathrm{Msg}(S)$. Similarly $bb_j\in\mathrm{Msg}(S)$. Since $S$ is generated by $aA_1\cup bA_2$ and every minimal generator belongs to any generating set, we have $\mathrm{Msg}(S)=aA_1\cup bA_2$.
\item First note that $aA_1$ and $bA_2$ are disjoint: if $aa_i=bb_j$, then $bb_j=a\cdot 0+b\cdot b_j$ is a representation with $b_j>0$, contradicting the fact above. Hence, by (1), $e(S)=|\mathrm{Msg}(S)|=|A_1|+|A_2|=e(S_1)+e(S_2)$. Since the least positive element of $S$ is necessarily a minimal generator,
\[
m(S)=\min\mathrm{Msg}(S)=\min(a\,m(S_1),\,b\,m(S_2)).
\]
\end{enumerate}
Parts (3) and (4) are Theorem~\ref{thm:delorme}; the standard facts on gluings can also be found in \cite[\S 8.3]{RGS09} and \cite[\S 1]{AGS13}.
\end{proof}
\end{lemma}

\section{Main results}\label{sec:main}

The following lemma is the core of the construction: for every embedding dimension $e\ge 2$, it provides a family of complete intersection numerical semigroups whose multiplicity equals the least possible value $2^{e-1}$, and the family contains semigroups with arbitrarily large non-generators.
\begin{lemma}\label{lem:cover}
For every integer $e\ge 2$ and every integer $x\ge 2^e$ there exists a complete intersection numerical semigroup $C_e(x)$ with $e(C_e(x))=e$, $m(C_e(x))=2^{e-1}$, and $x\in C_e(x)\setminus\mathrm{Msg}(C_e(x))$.

\begin{proof}
We argue by induction on $e$. Base case $e=2$: take $C_2(x)=\langle 2,3\rangle$. Then $e(C_2(x))=2$ and $m(C_2(x))=2=2^{2-1}$; every integer $x\ge 4$ is a sum of at least two positive elements of $\langle 2,3\rangle$, so $x\notin\mathrm{Msg}(C_2(x))$.

Inductive step: assume the statement holds for $e-1$ ($e\ge 3$), and set $m_0=2^{e-2}=m(C_{e-1}(\cdot))$. Fix $x\ge 2^e$ and split into two cases according to the parity of $x$.

Case 1: $x$ is even. Write $x=2r$. Then $r=x/2\ge 2^{e-1}$, so the induction hypothesis yields a complete intersection $U=C_{e-1}(r)$ with $e(U)=e-1$, $m(U)=m_0$, and $r\in U\setminus\mathrm{Msg}(U)$. Since $U$ has even multiplicity $m_0$ and $\gcd(U)=1$, the semigroup $U$ contains an odd minimal generator; adding sufficiently many copies of $m_0$ to this generator produces an odd element of $U$ which is not a minimal generator. Hence we may choose an odd element $n\in U\setminus\mathrm{Msg}(U)$ with $n>2m_0=2^{e-1}$.

Define $C_e(x)=2U+n\mathbb{N}$. Since $2\in\mathbb{N}\setminus\mathrm{Msg}(\mathbb{N})$, $n\in U\setminus\mathrm{Msg}(U)$, and $\gcd(2,n)=1$, this is a gluing of $\mathbb{N}$ and $U$. By Lemma~\ref{lem:gluing}, $C_e(x)$ is a complete intersection, and $e(C_e(x))=1+e(U)=e$, $m(C_e(x))=\min(n,2m_0)=2m_0=2^{e-1}$ (because $n>2m_0$).

Finally, $r\in U\setminus\mathrm{Msg}(U)$ means that $r=u+v$ for positive elements $u,v\in U$. Hence $x=2r=2u+2v\in 2U\subseteq C_e(x)$, and $2u,2v$ are positive elements of $C_e(x)$ strictly smaller than $x$; therefore $x$ is not a minimal generator.

Case 2: $x$ is odd. Put $n=x-2^{e-1}$. Since $x\ge 2^e$, we have $n\ge 2^{e-1}$. By the induction hypothesis, $U=C_{e-1}(n)$ is a complete intersection with $e(U)=e-1$, $m(U)=m_0$, and $n\in U\setminus\mathrm{Msg}(U)$. Define $C_e(x)=2U+n\mathbb{N}$. As in Case 1, this is a gluing of $\mathbb{N}$ and $U$: $2\in\mathbb{N}\setminus\mathrm{Msg}(\mathbb{N})$, $n\in U\setminus\mathrm{Msg}(U)$, and $\gcd(2,n)=1$ because $n$ is odd. Hence $C_e(x)$ is a complete intersection, and $e(C_e(x))=e$, $m(C_e(x))=\min(n,2m_0)=2m_0=2^{e-1}$, where the last equality holds because $n\ge 2^{e-1}=2m_0$.

It remains to show that $x$ is not a minimal generator. The element $m_0$ is a minimal generator of $U$, so $2m_0\in\mathrm{Msg}(2U)\subseteq\mathrm{Msg}(C_e(x))$, and $n\in\mathrm{Msg}(C_e(x))$. Since $x=n+2^{e-1}=n+2m_0$, the element $x$ is a sum of two positive elements of $C_e(x)$, hence $x\notin\mathrm{Msg}(C_e(x))$. This completes the induction.
\end{proof}
\end{lemma}

\begin{example}\label{ex:cover}
$C_3(9)=\langle 4,5,6\rangle$ and $C_4(17)=\langle 8,9,10,12\rangle$.
\end{example}

\begin{theorem}\label{thm:necessity}
If a complete intersection numerical semigroup $\Gamma$ satisfies $e(\Gamma)=e$ and $m(\Gamma)=m$, then $(e,m)=(1,1)$ or $e\ge 2$ and $m\ge 2^{e-1}$.

\begin{proof}
If $e(\Gamma)=1$, then $\Gamma=\mathbb{N}$ and $m(\Gamma)=1$. If $e(\Gamma)\ge 2$, Lemma~\ref{lem:lower-bound} gives $m(\Gamma)\ge 2^{e(\Gamma)-1}$. Hence every realized pair lies in the stated set.
\end{proof}
\end{theorem}

We now restate the main result and complete its proof.
\begin{theorem}\label{thm:main-section}
There exists a complete intersection numerical semigroup $\Gamma$ with $e(\Gamma)=e$ and $m(\Gamma)=m$ if and only if $(e,m)=(1,1)$ or $e\ge 2$ and $m\ge 2^{e-1}$. Equivalently, the set of realized pairs is

\[
\{(1,1)\}\cup\{(e,m)\in\mathbb{Z}^2 : e\ge 2,\ m\ge 2^{e-1}\}.
\]

\begin{proof}
The necessity was established in Theorem~\ref{thm:necessity}, so it remains to prove sufficiency. For $(1,1)$ take $\Gamma=\mathbb{N}$. For $e=2$ and $m\ge 2$, the two-generated semigroup $\Gamma=\langle m,m+1\rangle$ is a complete intersection with $e(\Gamma)=2$ and $m(\Gamma)=m$.

Assume $e\ge 3$ and $m\ge 2^{e-1}$. By Lemma~\ref{lem:cover} applied with $e-1$ in place of $e$, there is a complete intersection numerical semigroup $U=C_{e-1}(m)$ with

\[
e(U)=e-1,\qquad m(U)=2^{e-2},\qquad m\in U\setminus\mathrm{Msg}(U).
\]

Choose an integer $n\ge 2$ such that $\gcd(n,m)=1$ and $n\,2^{e-2}>m$; for instance, take a prime $n>\max\{2,m,m/2^{e-2}\}$. Define

\[
\Gamma=m\mathbb{N}+nU.
\]

Here $m\in U\setminus\mathrm{Msg}(U)$, $n\in\mathbb{N}\setminus\mathrm{Msg}(\mathbb{N})$ (because $n\ge 2$), and $\gcd(m,n)=1$, so $\Gamma$ is the gluing of $\mathbb{N}$ and $U$. Both factors are complete intersections; by Lemma~\ref{lem:gluing}, $\Gamma$ is a complete intersection, and

\[
e(\Gamma)=1+e(U)=e,\qquad m(\Gamma)=\min(m,n\,2^{e-2})=m,
\]

the minimum being $m$ because $n\,2^{e-2}>m$. Hence the stated set is exactly the set of realized pairs.
\end{proof}
\end{theorem}

Using SINGULAR we obtained the following explicit semigroups for small parameters $(e,m)$:
\begin{table}[htbp]
\centering
\caption{Explicit complete intersection numerical semigroups for small pairs $(e,m)$}\label{tab:examples}
\begin{tabular}{ll}
$(e,m)$ & minimal generating set \\ \hline
$(3,4)$ & $\langle 4,5,6\rangle$ \\
$(3,5)$ & $\langle 5,6,9\rangle$ \\
$(3,6)$ & $\langle 6,7,8\rangle$ \\
$(4,8)$ & $\langle 8,9,10,12\rangle$ \\
$(4,9)$ & $\langle 9,20,25,30\rangle$ \\
$(4,10)$ & $\langle 10,11,12,18\rangle$ \\
$(5,16)$ & $\langle 16,136,153,170,204\rangle$ \\
$(5,17)$ & $\langle 17,24,27,30,36\rangle$ \\
$(6,32)$ & $\langle 32,592,629,666,740,888\rangle$ \\
\hline
\end{tabular}
\end{table}

The last three rows are representative examples checked with SINGULAR during the verification: in each case, the toric ideal of $k[\![\Gamma]\!]$ is generated by exactly $e-1$ relations, confirming that $\Gamma$ is a complete intersection.

\section{Discussion and outlook}\label{sec:discussion}

The characterization shows that the lower bound $m\ge 2^{e-1}$ is the only obstruction: every pair above the boundary curve $m=2^{e-1}$ is realizable. The covering family $C_e(x)$ of Lemma~\ref{lem:cover} shows that the boundary itself is attainable for every $e$, and the gluing step $m\mathbb{N}+nU$ shows how to increase the multiplicity while keeping the embedding dimension fixed.

Every complete intersection numerical semigroup is symmetric, so every pair realized by a complete intersection must be realized by a symmetric semigroup. Rosales proved that symmetric semigroups realize exactly the pairs with $2\le e\le m-1$ together with the exceptional pairs $(1,1)$ and $(2,2)$ \cite{Ros01}. Our theorem shows that for $e=2,3$ the two realizable regions coincide, while for $e\ge 4$ complete intersections impose the strictly stronger constraint $m\ge 2^{e-1}>e+1$. In other words, the symmetric classification is only necessary, not sufficient, for our problem: the region $e+1\le m<2^{e-1}$ (nonempty for $e\ge 4$) is filled with symmetric semigroups but contains no complete intersection semigroup.

The proof is constructive and elementary. Natural follow-ups include: (i) formalizing the covering lemma and the main theorem with a proof assistant; (ii) studying the minimal Frobenius number or genus of semigroups realizing a given pair $(e,m)$; (iii) computing the same realizable regions for neighboring classes such as $\eta$-minimal semigroups or free semigroups.

\bibliographystyle{amsplain}
\bibliography{refs}

@misc{MS24,
  title = {Open problems on relations of numerical semigroups},
  author = {Moscariello, Alessio and Sammartano, Alessio},
  howpublished = {arXiv:2406.00790v2},
  year = {2026},
  note = {Published in Recent Progress in Ring and Factorization Theory, Springer Proc. Math. Stat. 477 (2025), DOI 10.1007/978-3-031-75326-8\_16; Problem 29, Section 7, p.~12}
}

@article{AGS13,
  title = {Constructing the set of complete intersection numerical semigroups with a given {Frobenius} number},
  volume = {24},
  journal = {Applicable Algebra in Engineering, Communication and Computing},
  publisher = {Springer},
  author = {Assi, A. and Garc\'{i}a-S\'{a}nchez, P. A.},
  year = {2013},
  number = {2},
  pages = {133--148},
  doi = {10.1007/s00200-013-0186-z},
  eprint = {1204.4258},
  archivePrefix = {arXiv}
}

@article{EHOOPK24,
  title = {On the cardinality of minimal presentations of numerical semigroups},
  volume = {7},
  journal = {Algebraic Combinatorics},
  author = {Elmacioglu, Ceyhun and Hilmer, Kieran and O'Neill, Christopher and Okandan, Melin and Park-Kaufmann, Hannah},
  year = {2024},
  number = {3},
  pages = {753--771},
  doi = {10.5802/alco.354}
}

@article{Del76,
  title = {Sous-mono{\"i}des d'intersection compl{\`e}te de $\mathbb{N}$},
  volume = {9},
  journal = {Annales scientifiques de l'{\'E}cole normale sup{\'e}rieure},
  author = {Delorme, Charles},
  year = {1976},
  number = {1},
  pages = {145--154},
  doi = {10.24033/asens.1307}
}

@book{RGS09,
  title = {Numerical Semigroups},
  author = {Rosales, J. C. and Garc\'{i}a-S\'{a}nchez, P. A.},
  series = {Developments in Mathematics},
  volume = {20},
  publisher = {Springer},
  address = {New York},
  year = {2009}
}

@article{Ros01,
  title = {Symmetric numerical semigroups with arbitrary multiplicity and embedding dimension},
  volume = {129},
  journal = {Proceedings of the American Mathematical Society},
  author = {Rosales, J. C.},
  year = {2001},
  number = {8},
  pages = {2197--2203},
  doi = {10.1090/s0002-9939-01-05819-1}
}

\end{document}